\documentclass[10pt, reqno, draft]{amsart}

\usepackage{amsfonts, mathrsfs}
\usepackage{amsthm}
\usepackage{amsmath}
\usepackage{graphicx}
\usepackage{amscd,amssymb,amsthm}

\title[Summations associated with GEEW distribution]
{Summations associated with Gamma exponentiated exponential Weibull distribution}

\author{Dragana Jankov Ma\v sirevi\'c}
\address{Department of Mathematics, University of Osijek, 31000 Osijek, Croatia}
\email{djankov@mathos.hr}

\author{Tibor K. Pog\'any}
\address{Faculty of Maritime Studies, University of Rijeka, Rijeka 51000, Croatia \\
Applied Mathematics Institute, \'Obuda University, 1034 Budapest, Hungary}
\email{poganj@pfri.hr} 



\newtheorem{lemma}{Lemma}
\newtheorem{theorem}{Theorem}
\newtheorem{corollary}{Corollary}

\newtheorem{remark}{Remark}

\begin{document}

\maketitle
\allowdisplaybreaks

\begin{abstract}
Considering the recently studied Gamma exponentiated exponential Weibull ${\rm GEEW}(\theta)$ probability
distribution \cite{PoganySaboor} surprising infinite summations are obtained for series which building blocks are special
functions like lower and upper incomplete Gamma, Fox--Wright Psi, Meijer $G$ or Whittaker $W$ functions.
\end{abstract}

\allowdisplaybreaks

\section{Introduction and preliminaries}

Adding parameters to an existing distribution enables one to obtain classes of more flexible distributions. Zografos and 
Balakrishnan \cite{Zogr_Bala:2009} introduced an interesting method for adding a new parameter to an existing distribution. The 
new distribution provides more flexibility to model various types of data. The baseline distribution has the survivor function 
$\overline{G}(x)=1-G(x)$. Then, the Gamma--exponentiated extended distribution has cumulative distribution function (CDF) $F(x)$ 
given by
   \[ F(x) = \frac{1}{\Gamma(\alpha)} \int_0^{-\log \overline{G}(x)}\,t^{\alpha-1}\,{\rm e}^{-t}\,{\rm d}t\,,\qquad
                   \alpha>0,\, x\in \mathbb R\,.\]
The related gamma--exponentiated extended probability density function (PDF) can be expressed in the following form:
   \[ f(x) = \frac{1}{\Gamma(\alpha)} \left(-\log \overline{G}(x)\right)^{\alpha-1}\,g(x)\,,\qquad  \alpha>0,\, x\in \mathbb R\,,\]
where $g = G'$. The so--called {\it regularized Gamma function}
   \[ Q(a,z) = \frac{\Gamma(a,z)}{\Gamma(a)}  = \frac1{\Gamma(a)}
               \int_z^\infty \, t^{a-1} {\rm e}^{-t}\, {\rm d} t, \qquad \Re(a)>0\, ,\]
where $\Gamma(a,x)$ denotes the familiar {\it upper incomplete Gamma function}. Both,  regularized Gamma and incomplete Gamma, 
are in-built in {\it Mathematica} under {\tt GammaRegularized[a, z]} and {\tt Gamma[a,z]} respectively.

We reformulate their model by applying the gamma--exponentiated technique \cite{Zogr_Bala:2009} in the following way. Choosing the 
baseline distribution's survivor function to be $\overline G(x) = 1-G(x) = \exp\{- h(x)\}$, where $h \colon \mathbb R_+ \mapsto
\mathbb R_+$ denotes a nonnegative Borel function. Accordingly, the rv $X$, defined on a standard probability space 
$(\Omega, \mathfrak F, \mathsf P)$, having CDF and PDF
   \begin{align}\label{J1}
       F(x) &= \left(1-Q\big(\alpha,\,h(x)\right)\,\mathsf 1_{\mathbb R_+}(x) \\ \label{J2}
       f(x) &= \frac{h'(x)}{\Gamma(\alpha)}\,h^{\alpha-1}(x)\,{\rm e}^{-h(x)}\,\mathsf 1_{\mathbb R_+}(x),
   \end{align}
we call {\it Gamma exponentiated $h$} distributed, noting this $X \sim {\rm GE}(\alpha, h)$. Here, and in what follows 
$\mathsf 1_A(x)$ denotes the characteristic function of the set $A$, that is $\mathsf 1_A(x) =1$ when $x\in A$ and equals 0 else. The 
well--known confluent hypergeometric form of the upper incomplete Gamma function enables to write 
   \[ F(x) = \frac{h^\alpha(x)}{\Gamma(\alpha+1)}\, {}_1F_1(\alpha; \alpha+1; -h(x))\,  \mathsf 1_{\mathbb R_+}(x) \, ,\]
where the {\it confluent hypergeometric function's} (or in other words {\it Kummer's function of the first kind}) series definition 
reads
   \[ {}_1F_1(a;b;z) = \sum_{n\geq 0} \frac{(a)_n}{(b)_n}\, \frac{z^n}{n!}\,. \]
The related {\it Mathematica} code is {\tt Hypergeometric1F1[a, b, z]}.

By choosing $h(x) = \lambda x+\beta x^k$, that is when the baseline survivor function $\overline{G}(x) = \exp\{-(\lambda 
x+\beta x^k\big)\}$, Pog\'any and Saboor introduced and discussed in detail the so--called Gamma--exponentiated exponential 
Weibull distribution ${\rm GEEW}(\theta), \theta = (\lambda, \beta, k, \alpha)>0$, see \cite[Introduction]{PoganySaboor}. The related 
CDF and PDF are
   \begin{align*} 
       F(x) &= \left(1-Q\big(\alpha,\,\lambda x+\beta x^k\big)\right)\,\mathsf 1_{\mathbb R_+}(x) \\
            &= \frac{(\lambda x+\beta x^k)^\alpha}{\Gamma(\alpha+1)}\, {}_1F_1\big(\alpha; \alpha+1; -(\lambda x+\beta x^k)\big)\,
               \mathsf 1_{\mathbb R_+}(x),\notag \\ 
       f(x) &= \frac1{\Gamma(\alpha)}\,\left(\lambda+ \beta \,k\,x^{k-1}\right)\,{\rm e}^{-\lambda\,x-\beta\,x^k}\,
               \left(\lambda x+\beta x^k\right)^{\alpha-1}\,\mathsf 1_{\mathbb R_+}(x),
   \end{align*}
respectively. Being our main goal to present summation formulae associated with ${\rm GEEW}(\theta)$, we need the following 
results regarding the moments of a rv $X \sim {\rm GEEW}(\theta)$ and to the a rv $Y \sim {\rm GE}(\alpha, h)$.

\begin{theorem}
Consider the rv $X \sim {\rm GE}(\alpha, h), \alpha>0$ defined by the {\rm CDF} \eqref{J1} and the related {\rm PDF} \eqref{J2}. 
Then the rv $Y = h(X) \sim {\rm Gamma}(\alpha, 1)$ and the related {\rm CHF} reads
   \begin{equation} \label{J3}
        \phi_Y(t) = (1-{\rm i}t)^{-\alpha},\quad t\in \mathbb R,
     \end{equation}
where ${\rm Gamma}(\alpha, 1)$ stands for the two--parameter Gamma distribution and $h \colon \mathbb R_+ \mapsto \mathbb R_+$ is 
a non--decreasing Borel function. Moreover, we have
   \[ \mathsf E Y^s = (\alpha)_s, \qquad \Re(s)>-\alpha . \]
\end{theorem}

\begin{proof} By definition, the characteristic function (CHF) of an rv $X$ is the Fourier transform of the related PDF $f_X(x)$:
   \[ \phi_Y(t) = \mathsf E\,{\rm e}^{{\rm i}t Y} = \int_{\mathbb R} {\rm e}^{{\rm i}t h(x)} f(x)\, {\rm d}x, \qquad t \in \mathbb R. \]
Accordingly, substituting $h(x) = y$ we confirm the statement \eqref{J3}. However, we immediately recognize 
$\phi_Y(t) = (1-{\rm i}t)^{-\alpha}$ as the CHF of a rv possessing ${\rm Gamma}(\alpha, 1)$ distribution. The rest is obvious. 
\end{proof}

The Lambert $W$-function is the inverse function of $W \mapsto W{\rm e}^W$. Its principal branch $W_{\rm P}$ is the solution of 
$W{\rm e}^W = x$ for which $W_{\rm P}(x) \geq W_{\rm P}(-{\rm e}^{-1})$. This function is implemented in {\it Mathematica} 
as {\tt ProductLog[z]}. We are interested in $W_{\rm P}$ exclusively for $x\geq 0$, where it is single--valued and monotone 
increasing, see \cite{NIST1}.

Any nondecreasing function $h$ possesses a so-called {\em generalized inverse}
   \[ \mathfrak h^-(x) := \inf \{ t \in \mathbb R_+ \colon h(t) \ge x\}, \qquad  t\in \mathbb R_+\, ,\]
with the convention that $\inf\, \emptyset = \infty$. Moreover, if $h$ is strong monotone increasing, that is $h(x)<h(y)$ for all 
$0<x<y$, then $\mathfrak h^-$ coincides with the `ordinary' inverse $h^{-1}$.

Now, following the lines of the previous proof, we arrive at

\begin{theorem}
Consider random variables $X \sim {\rm GE}(\alpha, h)$ and $\Upsilon = h(X)\exp\{ \sigma h(X)\}, \sigma \geq 0$. Then
   \begin{equation} \label{J60}
        \Upsilon \sim {\rm GE}\left(\alpha, \mathfrak h^{-}\left(\sigma^{-1}\, W_{\rm P}(\sigma x)\right)\right) \, .
     \end{equation}
Further, for all $-\alpha<s<\sigma^{-1}$ it holds true
   \begin{equation*} 
        \mathsf Eh^s(X)\exp\{\sigma s\, h(X)\} = \frac{(\alpha)_s}{(1-\sigma s)^{\alpha+s}}\, ,
     \end{equation*}
whenever $h \colon \mathbb R_+ \mapsto \mathbb R_+$ is a
nondecreasing Borel function.
\end{theorem}

\begin{proof} The rv $X \sim {\rm GE}(\alpha, h)$ possesses CDF $F_X$ in the form \eqref{J1}. When $\sigma = 0$, then $\Upsilon
\equiv X$. Letting $\sigma >0$, the PDF $F_\Upsilon$ of the rv
$\Upsilon = h(X)\exp\{ \sigma h(X)\}$ becomes
   \begin{align*}
        F_\Upsilon(x) &=  \mathsf P \{ \Upsilon < x\} \\
				              &= \mathsf P \{ \sigma h(X)\exp\{ \sigma h(X)\} < \sigma x\} \\
							        &= \mathsf P \{ h(X) < \sigma^{-1}W_{\rm P}(\sigma x)\} \\
                      &= \mathsf P \left\{ X < \mathfrak h^-\big(\sigma^{-1}W_{\rm P}(\sigma x)\big)\right\} \\
                      &= F_X\left[\mathfrak h^-\big(\sigma^{-1}W_{\rm P}(\sigma x)\big)\right]\, ,
     \end{align*}
which is equivalent to the first assertion \eqref{J60}.

Next, in turn
   \[ \mathsf Eh^s(X)\exp\{\sigma s\, h(X)\} = \frac1{\Gamma(\alpha)} \int_0^\infty h^{\alpha + s-1}(x)\,
               {\rm e}^{-(1-\sigma s)\, h(x)}\, {\rm d}h(x)\, ,\]
where the convergence of the integral is controlled by the condition $\sigma s<1$, being $h$ nondecreasing and positive at
the infinity. 
\end{proof}

The following straightforward consequence of previous results we will need in the summation derivations.

\begin{corollary} Let $X \sim {\rm GEEW}(\theta), \, \theta = (\lambda, \beta, k, \alpha)>0$. Then
consider the rv $\lambda X + \beta X^k \sim {\rm Gamma}(\alpha, 1)$
and for all $\Re(s)>-\alpha$ we have
   \begin{equation} \label{M01}
       \mathsf E\, (\lambda X + \beta X^k)^s = \frac{\Gamma(\alpha+s)}{\Gamma(\alpha)}
                                                  = (\alpha)_s\,. \medskip
      \end{equation}
\end{corollary}

\section{By--product Summation Formulae for Hypergeometric Type Special Functions}

The following interesting facts turn out to be a consequences of the exponentiation procedure. Employing the derived expressions 
which concerns computation formulae of higher transcendental function terms of hypergeometric type, that is lower and upper 
incomplete Gamma functions $\gamma(a, z), \Gamma(a, z)$; confluent (unified confluent) Fox--Wright generalized hypergeometric function 
${}_1\Psi_0 ({}_1\Psi_0^*)$; Meijer $G_{13}^{31}$ function and Whittaker function of the second kind $W_{\nu, \mu}$.

Certain attractive special cases of the Theorem 3, which are evidently not so obvious corollaries of \eqref{N1} below, are our 
main results. In turn, let us recall in short the definitions of the above mentioned special functions. Firstly,
   \[ {}_p\Psi_q^*\Big[ \begin{array}{c} (a, A)_p\\ (b, B)_q \end{array} \Big|\, z\,\Big] = \sum_{n=0}^\infty
                              \frac{\prod_{j=1}^p (a_j)_{A_jn}}
                              {\prod_{j=1}^q (b_j)_{B_jn}}\, \frac{z^n}{n!} \]
is the {\em unified variant of the Fox--Wright  generalized hypergeometric function} with $p$ upper and $q$ lower parameters; 
$(a,A)_p$ denotes the parameter $p$--tuple $(a_1, A_1), \cdots, (a_p, A_p)$ and $a_j \in \mathbb C$, $b_i \in \mathbb C 
\setminus \mathbb Z_0^-$, $A_j, B_i>0$ for all $j=\overline{1,p}, i=\overline{1,q}$, while the series converges for suitably 
bounded values of $|z|$ when
   \[ \Delta_{p,q} := 1 - \sum_{j=1}^pA_j + \sum_{j=1}^qB_j >0\, .\]
In the case $\Delta_{p,q} = 0$, the convergence holds in the open disc $|z|<\beta = \prod_{j=1}^qB_j^{B_j} \cdot \prod_{j=1}^pA_j^{-A_j}$. 
The convergence condition $\Delta_{1,0} = 1-A_1>0$ is of special interest for us.

We point out that another definition of the Fox--Wright function ${}_p\Psi_q[z]$ (consult the monographs \cite{KST}, \cite{MS3}) 
contains Gamma functions instead of the here used generalized Pochhammer symbols. However, these two functions differ only up to 
constant multiplying factor, that is
   \[ {}_p\Psi_q \Big[ \begin{array}{c} (a, A)_p\\ (b, B)_q \end{array} \Big|\, z\,\Big] = \frac{\prod_{j=1}^p \Gamma(a_j)}
      {\prod_{j=1}^q\Gamma(b_j)}\, {}_p\Psi_q^* \Big[ \begin{array}{c} (a, A)_p\\ (b, B)_q \end{array} \Big|\, z\,\Big]\, .\]
The unification's motivation is clear; for $A_1 = \cdots = A_p = B_1 = \cdots = B_q = 1$, ${}_p\Psi_q^*[z]$ one reduces exactly 
to the generalized hypergeometric function ${}_pF_q[z]$.

Further, the symbol $G_{p,q}^{m,n}( \cdot |\, \cdot)$ denotes Meijer's $G-$function \cite{MEIJ} defined in terms of the  
Mellin--Barnes integral reads
   \[ G_{p,q}^{m,n}\Big( z \,\Big| \begin{array}{c} a_1, \cdots, a_p\\b_1, \cdots, b_q \end{array}\Big)
               =  \frac1{2\pi{\rm i}}\int_{\mathfrak C}\frac{\prod_{j=1}^{m}\Gamma(b_j-s) \prod_{j=1}^{n}\Gamma( 1-a_j + s)}
             {\prod_{j=m+1}^q \Gamma( 1-b_j + s) \prod_{j=n+1}^{p}\Gamma(a_j - s)}\, z^s\,ds, \]
where $0\le m\le q,\, 0\le n\le p$ and the poles $a_j, b_j$ are such that no pole of $\Gamma(b_j - s), j=\overline{1,m}$ coincides 
with any pole of  $\Gamma(1-a_j+s), j=\overline{1,n}$; i.e. $a_k-b_j \not\in \mathbb N$, while $z \neq 0$. $\mathfrak C$ is a 
suitable integration contour which startes at $-{\rm i}\infty$ and goes to ${\rm i}\infty$ separating the poles of $\Gamma(b_j - s), 
j=\overline{1,m}$ which lie to the right of the contour, from all poles of $\Gamma(1-a_j+s), j=\overline{1,n}$, which lie to the left of 
$\mathfrak C$. The integral converges if $\delta = m+n - \tfrac12(p+q)>0$ and $|{\rm arg} (z)| <\delta \pi$, see \cite[p. 143]{LUKE} 
and \cite{MEIJ}. Let us mention that the $G$ function's Mathematica code reads
   \[ \texttt{MeijerG[}\{\{a_1,...,a_n\}, \{a_{n+1},...,a_p\}\}, \{\{b_1,...,b_m\}, \{b_{m+1},...,b_q\}\}, z\texttt{]}.\]
Finally, we formulate a further special summation, where the moment $\mathsf E\, X^r$ we express in terms of the Whittaker function of 
the second kind $W_{a, b}(z)$ (see \cite[\S 13.14.]{NIST1}). One of the numerous connecting formulae close to our recent 
considerations reads as follows:
   \begin{align*}
      W_{a, b}(z) &= {\rm e}^{-\frac z2} z^{1/2 + b} \left( \dfrac{z^{-2b} \Gamma(2b)}{\Gamma(b-a+\frac12)}\, 
                     {}_1F_1\left(-b-a+\frac12; 1-2b; z\right) \right.\\
                  &+ \left. \dfrac{ \Gamma(-2b)}{\Gamma(-b-a+\frac12)}\,{}_1F_1\left(b-a+\frac12; 1+2b; z\right)\right), 
                     \qquad 2b \not\in \mathbb Z. 
   \end{align*}
The associated {\it Mathematica} code is $\texttt{WhittakerW[}a,b,z\texttt{]}$. 

\begin{lemma}{\rm \cite[Lemma 3.1]{PoganySaboor}} For all positive $(\mu, a, \nu, \rho)$, for which $\rho + 
\nu^{-1}(\ell + \mu) \not\in \mathbb N$ when $\ell\in \mathbb N$, we have
   \begin{align} \label{M02}
        I_\mu(a, \nu, \rho) &= \int_0^\infty x^{\mu-1} (1+ ax^\nu)^\rho\, {\rm e}^{-x}\,{\rm d}x \nonumber \\
                                 &= \sum_{n \geq 0} \frac{(-1)^n\,(-\rho)_n}{n!}\,\left\{ a^n\, \gamma\big(\mu + \nu\, n, a^{-1/\nu}\big)
                                                      + a^{\rho-n}\, \Gamma\big(\mu + \nu(\rho- n), a^{-1/\nu}\big)\right\},
     \end{align}
where $\gamma(a, z) = \Gamma(a) - \Gamma(a, z), \Re(a)>0$ signifies the {\it lower incomplete Gamma function}.
\end{lemma}

Moreover, specifying $\rho \in\mathbb N_0$ in \eqref{M02} $\mu, a, \nu$ remain positive, $I_\mu(a, \nu, \rho)$ one reduces to 
a polynomial in $a$ of ${\rm deg}(I_\mu) = \rho$:
   \begin{equation*} 
        I_\mu(a, \nu, \rho) = \Gamma(\mu) \, {}_2\Psi_0^* \Big[ \begin{array}{c} (-\rho, 1),\, (\mu, \nu) \\
                                  - \end{array}\Big| -a \Big]\,.
     \end{equation*}
Now, we formulate a set of summation results.

\begin{theorem} {\bf a.} For all $\alpha \in \mathbb R_+ \setminus \mathbb N$, denoting $a_0 = \left(\beta\lambda^{-k}\right)^{-1/(k-1)}$, 
we have
   \begin{align} \label{Q11}
         &\frac{\sin(\pi\alpha)}{\pi\alpha}\sum_{m,n \geq 0}\, \frac{(-\beta\lambda^{-k})^m}{m!}\,\frac{(-1)^n\,\Gamma(1-\alpha+n)}{n!}
                        \Bigg\{ \left(\frac\beta{\lambda^k}\right)^n\, \Big[ \gamma\big(\alpha +1 + km+(k-1)n, a_0\big) \nonumber \\
             &\qquad \quad+ \dfrac{\beta\,(k+1)}{\lambda^k}\,\gamma\big(\alpha  + k(m+1)+(k-1)n, a_0\big)
                                 + \frac{\beta^2 k}{\lambda^{2k}} \, \gamma\big(\alpha+k(m+2)-1+(k-1)n, a_0\big)\Big] \nonumber \\
             &\qquad \quad + \left(\frac\beta{\lambda^k}\right)^{\alpha-n}\,\left[ \frac{\lambda^k}{\beta}
                                   \Gamma\big(k(\alpha-1+ m)+2-(k-1)n, a_0\big) \right. \nonumber \\
             &\qquad  \quad \left. + (k+1)\,\Gamma\big(k(\alpha+m)+1-(k-1)n, a_0\big)+\frac{\beta\,k}{\lambda^k}\,
                                   \Gamma\big(k(\alpha+m+1)-(k-1)n, a_0\big)\right]\Bigg\}\, = 1\, .
     \end{align}

\noindent {\bf b.} For all $\alpha \in \mathbb N;\,(\lambda, \beta, k)>0$  we have 
   \begin{align*} 
        & \sum_{n=0}^{\alpha-1} \frac{(1-\alpha)_n}{n!}\, \left(-\frac\beta{\lambda^k}\right)^n\left\{
                      \Gamma(\alpha+1+kn)\,{}_1\Psi_0^* \Big[ \begin{array}{c} (\alpha+1+(k-1)n, k) \\ 
                    - \end{array}\Big| -\frac\beta{\lambda^k} \Big]\right. \\
            &\qquad + \frac{1}{\lambda^{k}\Gamma(k)} \,
                      {}_1\Psi_0 \Big[ \begin{array}{c}(\alpha+k(1+n), k) \\ -\end{array}\Big| -\frac\beta{\lambda^k} \Big]
                    + \frac{\beta (k)_{\alpha+kn}}{\lambda^k} \,
                      {}_1\Psi_0^* \Big[ \begin{array}{c} (\alpha+k+(k-1)n, k) \\ - \end{array}\Big| -\frac\beta{\lambda^k} \Big] 
											\nonumber \\
            &\qquad + \left. \frac{\beta }{\lambda^{2k}\,\Gamma(2k-1)} \,
                      {}_1\Psi_0 \Big[ \begin{array}{c}(\alpha+k(2+n)-1, k) \\ -\end{array}\Big| -\frac\beta{\lambda^k} \Big]\right\}
                    = \alpha. \medskip
     \end{align*}

\noindent {\bf c.} For all $\lambda>0, u >0, \alpha>0$ we have
     \begin{align*} 
        & \frac{\sin(\pi\alpha)}{\pi\alpha}\, \sum_{n\geq 0} \,\frac{\left(-\dfrac{u}{\lambda^2}\right)^n}{n!}\,
                   \Bigg\{ G_{13}^{31}\Bigg( \frac{(u\lambda)^2}4\Bigg|
                   \begin{array}{c} \tfrac{1-\alpha-3n}{2} \\  \tfrac{1+\alpha-3n}{2},\, 0,\, \tfrac12 \end{array}\Bigg)
                 + \frac3{\lambda^3}\, G_{13}^{31}\Bigg( \frac{(u\lambda)^2}4\Bigg|
                   \begin{array}{c} -\tfrac{1+\alpha+3n}{2} \\ \tfrac{\alpha-1-3n}{2},\, 0,\, \tfrac12 \end{array}\Bigg)\nonumber \\
        & \qquad + \frac{1}\lambda\, G_{13}^{31}\Bigg( \frac{(u\lambda)^2}4\Bigg|
                   \begin{array}{c} -\tfrac{\alpha+1+3n}{2} \\ \tfrac{\alpha-1-3n}{2},\, 0,\, \tfrac12 \end{array}\Bigg)
                 + \frac3{\lambda^4}\, G_{13}^{31}\Bigg( \frac{(u\lambda)^2}4\Bigg|
                   \begin{array}{c} -\tfrac{\alpha+3+3n}{2} \\ \frac{\alpha-3-3n}{2},\, 0,\, \tfrac12 \end{array}\Bigg) \Bigg\}
                                 = \frac{2\sqrt{\pi}}{u^{\alpha+1}\lambda^2}\, .\medskip
     \end{align*}

\noindent {\bf d.} For all $\lambda>0, b>0, \alpha>0$, we have
   \begin{align*}
        &\lambda\, \sum_{n \geq 0} \frac{(b\lambda)^n}{n!\,\Gamma(\alpha+1+2n)}\,
                        \Bigg\{ W_{\frac{\alpha-1}2-n, \frac{3\alpha}2+n}\left(b\lambda\right)
                    + 3\,\frac{W_{\frac{\alpha}2-1-n, \frac{3\alpha+1}2+n}\left(b\lambda\right)}{\lambda^2\sqrt{b\lambda}(\alpha+1+2n)}\,
                      \nonumber \\
            &\qquad\qquad + 2\, \frac{W_{\frac{\alpha-3}2-n, \frac{3\alpha}2+1+n}\left(b\lambda\right)}
                  {\lambda^5b(\alpha+1+2n)(\alpha+2+2n)} \Bigg\}
                    = \alpha\,\left( \frac\lambda{b}\right)^{\tfrac{\alpha+1}2}\,e^{-\frac12 b\lambda}\, \, .
     \end{align*}
\end{theorem}

\begin{proof} \,{\bf a.} Let us consider the rv $X \sim {\rm GEEW}(\theta), \theta = (\lambda, \beta, k, \alpha)$ having {\rm PDF}
   \begin{align*}
      f(x) &= \frac{\lambda^\alpha}{\Gamma(\alpha)} x^{\alpha-1}\left(1 + \frac\beta\lambda x^{k-1}\right)^{\alpha-1}\,
              {\rm e}^{-(\lambda x + \beta x^k)}\, \mathsf 1_{\mathbb R_+}(x) \nonumber \\
           &\qquad + \frac{\lambda^{\alpha-1} \beta k}{\Gamma(\alpha)} x^{k+\alpha-2}
              \left(1 + \frac\beta\lambda x^{k-1}\right)^{\alpha-1}\, {\rm e}^{-(\lambda x + \beta x^k)}\, \mathsf 1_{\mathbb R_+}(x).
   \end{align*}
By expanding the exponential term $\exp\{-\beta x^k\}$ into the Maclaurin series and using \eqref{M02} we get $\mathsf EX^r$ as the linear 
combination of two $I_\mu(a, \nu, \rho)$ integrals:
   \begin{equation} \label{N1} 
      \mathsf EX^r = \frac{\lambda^{-r}}{\Gamma(\alpha)} \sum_{m\geq 0} \dfrac{(-\beta)^m}{\lambda^{km} \, m!} 
                     \left\{ I_{r+\alpha+km}\left( \frac\beta{\lambda^k}, k-1, \alpha-1\right) 
                   + \frac{\beta k}{\lambda^k}\, I_{r+\alpha+k(m+1)-1}\left( \frac\beta{\lambda^k}, k-1, \alpha-1\right)\right\} .
   \end{equation}
Now, highlighting the case $s=1$ of \eqref{M01} which leads to $\lambda \,\mathsf E\,X + \beta \mathsf E\, X^k = \alpha$ and 
taking $r=1, k$ above in \eqref{N1} respectively, we deduce the formula \eqref{Q11}. \medskip

\noindent As to the assertions {\bf b.} it is enough the point out the result \cite[Theorem 2]{PoganySaboor} 
   \begin{align} \label{OX1}
        &\mathsf EX^r = \frac{(r)_\alpha}{\lambda^r} \sum_{n=0}^{\alpha-1} \frac{(1-\alpha)_n(r+\alpha)_{kn}}{n!}\,
                        {}_1\Psi_0^* \Big[ \begin{array}{c} (r+\alpha+(k-1)n, k) \\ - \end{array}\Big| -\frac\beta{\lambda^k} \Big]
                        \, \left(-\frac\beta{\lambda^k}\right)^n \nonumber \\
               &\quad + \frac{(r+k-1)_\alpha}{\lambda^{r+k}} \sum_{n=0}^{\alpha-1} \frac{(1-\alpha)_n(r+\alpha+k-1)_{kn}}{n!}\,
                        {}_1\Psi_0^* \Big[ \begin{array}{c}(r+\alpha-1+k(n+1), k) \\ -\end{array}\Big| -\frac\beta{\lambda^k} \Big]
                        \, \left(-\frac\beta{\lambda^k}\right)^n\, ,
   \end{align}
valid for a rv $X \sim {\rm GEEW}(\theta)$ where $\alpha \in \mathbb N;\,(\lambda, \beta, k)>0$ and for all 
$r> \max\{-\alpha, 1-\alpha -k\}$. Repeating the proving procedure of {\bf a.} by setting $r=1, k$ in \eqref{OX1}, we arrive at 
the stated formula {\bf b.} \medskip

\noindent {\bf c.} Similarly, having in mind \cite[Theorem 3]{PoganySaboor} constituted for a rv 
$X \sim {\rm GEEW}(\lambda,(u\lambda)^{-2},3,\alpha)$, for which, under constraints $\lambda>0, u >0, \alpha>-1$ and for all 
$r> -1$, there holds
   \begin{align*}
        \mathsf EX^r &= \frac{u^{r+\alpha}\lambda^r\, \sin(\pi\alpha)}{2\,\pi^{\frac32}}
                \sum_{n\geq 0} \Bigg\{ G_{13}^{31}\Bigg( \frac{(u\lambda)^2}4\Bigg|
                \begin{array}{c} 1-\tfrac{r+\alpha+3n}{2} \\ 1-\frac{r-\alpha+3n}{2},\, 0,\, \tfrac12 \end{array}\Bigg) \nonumber \\
        &\qquad \qquad \qquad + \frac3{\lambda^3}\, G_{13}^{31}\Bigg( \frac{(u\lambda)^2}4\Bigg|
                \begin{array}{c} -\tfrac{r+\alpha+3n}{2} \\ \frac{\alpha-r-3n}{2},\, 0,\, \tfrac12 \end{array}\Bigg)
                \Bigg\} \,\dfrac{\left(-\dfrac u{\lambda^2}\right)^n}{n!} \, ,
     \end{align*}
we infer the stated summation which involves Meijer $G$ function terms.  \medskip   

\noindent Finally, statement {\bf d.} we achieve employing \cite[Theorem 4]{PoganySaboor}. When $X \sim {\rm GEEW}(\lambda, 
(b\lambda)^{-1}, 2, \alpha)$, then for all $\lambda>0, b >0, \alpha>-1$ and for all $r> -1$, we have
   \begin{align*} 
        \mathsf EX^r &= \left( \frac b\lambda\right)^{\tfrac{r+\alpha}2}\,e^{\frac12 b\lambda}\,
                            \sum_{n \geq 0} \frac{(b\lambda)^n}{n!\,\Gamma(r+\alpha+2n)}\,
                                            \Bigg\{ W_{\tfrac{\alpha-r}2-n, \tfrac{3\alpha+r-1}2+n}\left(b\lambda\right) \nonumber \\
                         &\qquad + \frac2{\lambda^2\sqrt{b\lambda}(r+\alpha+2n)}\,
                                          W_{\tfrac{\alpha-r-1}2-n, \tfrac{3\alpha+r}2+n}\left(b\lambda\right) \Bigg\}\, .
     \end{align*}
Now, the same treatment of the restricted formula \eqref{M01} under $s=1$ finishes the proof. 
\end{proof} 

\begin{remark} 
{\rm Similar summations neither contain the celebrated formula collection by Hansen \cite{Hansen} nor the comprehensive and 
exhaustive classical monograph series Prudnikov {\it et. al.} \cite{PBM}.}
\end{remark} 

\begin{remark} 
{\rm Considering the result of Corollary 1 for $s \in \mathbb N_2 = \{2, 3, \cdots\}$, expanding the 
left--hand--side expression in \eqref{M01} into a binomial sum, we can conclude a set of further summation results following the traces 
of the proof of Theorem 2. However, using into account this simple derivation procedure we lose the elegance in the resulting 
very complicated, hardly reducible resulting sum--product expressions. }  
\end{remark}

\end{document}